\documentclass[12 pt]{amsart}
\usepackage{amssymb, amsmath, amsthm}
\usepackage{enumitem}
\usepackage{hyperref}
\usepackage{bbm}
\usepackage{float}
\usepackage{caption}
\usepackage{tikz-cd}
\usepackage{bbm}
\usepackage{stmaryrd}
\usepackage{mathtools}

\DeclareMathOperator{\im}{im}
\DeclareMathOperator{\Pic}{Pic}
\DeclareMathOperator{\Spec}{Spec}

\DeclareMathOperator{\fieldchar}{char}

\DeclareMathOperator{\HH}{H}
\DeclareMathOperator{\Hom}{Hom}

\DeclareMathOperator{\Aut}{Aut}

\newcommand{\mg}[1]{\mathcal{M}_{#1}}
\newcommand{\mgbar}[1]{\overline{\mathcal{M}}_{#1}}

\newtheorem{theorem}{Theorem}[section]
\newtheorem{lemma}[theorem]{Lemma}
\newtheorem{proposition}[theorem]{Proposition}
\newtheorem{corollary}[theorem]{Corollary}
\newtheorem*{theorem*}{Theorem}

\theoremstyle{definition}
\newtheorem{definition}[theorem]{Definition}
\newtheorem{note}[theorem]{Note}

\newtheorem{example}[theorem]{Example}
\newtheorem{remark}[theorem]{Remark}

\title[Brauer groups of stacky curves]{Brauer groups of tame stacky curves and their $\mu_r$-gerbes}
\author{Martin Bishop}
\date{}

\begin{document}
\begin{abstract}
    We fit the Brauer group of a $\mu_r$-gerbe over
    a (possibly arbitrarily singular) tame stacky curve into an exact sequence and give characterizations for
    when it is short exact and conditions for when it splits. We also
    give a precise formula for the Brauer group of a $\mu_r$-gerbe
    in the case that the base curve is smooth.
    \end{abstract}

\maketitle

\section{Introduction}
To date, a significant portion of the current literature on Brauer groups
has been dedicated to the following picture:
given a specific smooth stacky curve $\mathcal C$ and a specific
$\mu_r$-gerbe $\mathcal G$ over $\mathcal C$,
compute the Brauer group of $\mathcal G$. See for instance
\cite{AM20}, with further progress in \cite{DLP25, Shi19}, which studied the Brauer group of
$\mg{1,1}$ -- a $\mu_2$-gerbe over a stacky curve -- and laid the foundations
for modern work on Brauer groups of stacks. See also \cite{ABJJL24}, which studied the
Brauer group of $\mathcal Y_0(2)$, also a $\mu_2$-gerbe over a stacky curve.
Most work has focused, very successfully, on generalizing the
base scheme $S$ for a fixed $\mathcal C$ and $\mathcal G$.

In this paper, we instead approach this problem from the opposite direction: we fix our base
$S$ as an algebraically closed field, and then allow $\mathcal C$ to be an arbitrary tame
stacky curve and $\mathcal G\rightarrow\mathcal C$ to be an arbitrary $\mu_r$-gerbe, with no smoothness
requirements and no requirements on the possible (tame) stabilizers.
To this end, our primary results are: complete computations
of the Brauer groups of tame stacky curves; complete computations of
the Brauer groups of $\mu_r$-gerbes over \textit{smooth} tame stacky curves; and
an exact sequence for the Brauer group of a $\mu_r$-gerbe over
a tame stacky curve along with precise conditions for when it is
short-exact and conditions for when it splits.

The main inspiration for, and the fundamental technical basis of, this work is the results
in the note \cite{Mei18}. Inspiration was also drawn from \cite{AM20, Lie11, Lop23, Shi21}.
While preparing this paper, the author learned of the work \cite{Ach24} which
uses a similar approach (leveraging the results of Section \ref{Meier section}
in conjunction with spectral sequences), although the papers operate on different levels of generality.
The paper \cite{Ach24} makes weaker assumptions on the base stack and the \'{e}tale-ness of stabilizers,
offering techniques which work for a wide variety of stacks provided they are \textit{locally Brauerless}
(\cite[Definition 4.13]{Ach24}). In the case of abelian stabilizers, this implies that the \'{e}tale
part must be cyclic (\cite[Remark 4.10]{Ach24}).

In contrast, this paper works in the more restrictive setting
of one-dimensional Deligne-Mumford stacks over algebraically closed fields, but with these stronger
assumptions is able to state results for stacks which are highly singular and \textit{not} locally Brauerless
(see Example \ref{stacky Tsen example} for an interesting example of a stack which is not
locally Brauerless).
In short,
we insist that our stabilizers must be tame and \'{e}tale, but impose no other conditions. Whereas
being locally Brauerless allows for non-\'{e}tale stabilizers but places strong requirements on
the tame \'{e}tale part. In this way, these two papers work in complementary levels of generality.
The research in this paper was conducted independently
of \cite{Ach24}.

\subsection{Main results}\label{main results}
Our main result is the following thorough classification of Brauer groups of $\mu_r$-gerbes over
tame stacky curves. Note that
here $\mathcal C$ is allowed to be arbitrarily singular.

\begin{theorem*}[Theorem \ref{main Brauer statement}]
Suppose that $\mathcal C$ is a tame separated one-dimensional Deligne-Mumford stack
with trivial generic stabilizer
of finite type over an
algebraically closed field $\mathbbm k$. Let
$f:\mathcal C\rightarrow C$ be the coarse moduli space of $\mathcal C$.
Suppose that $p_1,\dots,p_n$ are the geometric points
in $C$ whose preimage in $\mathcal C$ have non-trivial stabilizers $G_1,\dots,G_n$.
Let $\pi:\mathcal G\rightarrow\mathcal C$ be a
$\mu_r$-gerbe, and denote the stabilizers in $\mathcal G$ over $p_1,\dots,p_n$ as $E_1,\dots,E_n$.
Let $G_i$ and $E_i$ act trivially on $\mathbbm k^{\times}$ for each $i$.
Then the following sequence is exact, where the left cohomology is group cohomology and the right two
are \'{e}tale cohomology:
$$
\bigoplus_{i=1}^n\HH^2(G_i,\mathbbm k^{\times})\rightarrow
\HH^2(\mathcal G,\mathbb G_m)\rightarrow
\HH^1(\mathcal C,\mathbb Z/r).
$$
Moreover,
\begin{enumerate}
\item the left map is injective if and only if $\mathcal G$ is a root gerbe over $\mathcal C$,
\item $\mathcal G$ is a root gerbe over $\mathcal C$ if and only if it is a root gerbe on each fiber
$BE_i\rightarrow BG_i$,
\item the right map is surjective if and only if the inflation map
$$
\inf:\HH^3(G_i,\mathbbm k^{\times})\rightarrow\HH^3(E_i,\mathbbm k^{\times})
$$
is injective for all $i=1,\dots,n$,
\item if $\gcd(r,\left|G_i\right|)=1$ for all $i=1,\dots,n$, then:
\begin{enumerate}
\item $\pi:\mathcal G\rightarrow\mathcal C$
is automatically a root gerbe,
\item the inflation map is automatically injective, and
\item we have the following splitting
$$
\HH^2(\mathcal G,\mathbb G_m)=\HH^1(C,\mathbb Z/r)\oplus\bigoplus_{i=1}^n\HH^2(G_i,\mathbbm k^{\times}),
$$
\end{enumerate}
\item more generally,
in the situation that the sequence is right-exact, to show that the sequence splits it suffices
to show that on each fiber $BE_i\rightarrow BG_i$ there is a section
$$
\begin{tikzcd}
\HH^2(G_i,\mathbbm k^{\times})\arrow[r, "\inf"] & \HH^2(E_i,\mathbbm k^{\times})\arrow[l, dashed,
bend right].
\end{tikzcd}
$$
\end{enumerate}
\end{theorem*}

In the case where $\mathcal C$ is smooth, this vastly simplifies into the following
(the interested reader can find more cases in Section \ref{Leray for pi}).

\begin{theorem*}[Corollary \ref{weak smooth corollary}]
Suppose that $\mathcal C$ is a smooth separated tame one-dimensional Deligne-Mumford stack
with trivial generic stabilizer
of finite type over an algebraically closed field $\mathbbm k$.
Let $\pi:\mathcal G\rightarrow\mathcal C$ be a $\mu_r$-gerbe.
Then
$$
\HH^2(\mathcal G,\mathbb G_m)=\HH^1(\mathcal C,\mathbb Z/r).
$$
\end{theorem*}

We also present a comprehensive computation
of the cohomology of $\mathbb G_m$ on $\mathcal C$, with no smoothness requirements.

\begin{theorem*}[Proposition \ref{C G_m cohomology}]
Let $\mathcal C$ be a separated tame one-dimensional Deligne-Mumford stack with trivial
generic stabilizer
of finite type over an algebraically closed field $\mathbbm k$.
Let $G_1,\dots,G_n$ be the non-trivial stabilizers
of geometric points of $\mathcal C$.
Then for all $k\geq 2$,
$$
\HH^k(\mathcal C,\mathbb G_m)=\bigoplus_{i=1}^n\HH^k(G_i,\mathbbm k^{\times}),
$$
where, for even $k$, this is concentrated at the singular points of $\mathcal C_{red}$.
\end{theorem*}

Along the way, we also collect multiple characterizations of when a $\mu_r$-gerbe
over a stacky curve is a root gerbe
(see Proposition \ref{root injectivity} and Corollaries \ref{root fiber} and \ref{R^2 injectivity}).

\subsection{Structure of the paper}
We develop the necessary background on gerbes in Section \ref{gerbe background}.
Section \ref{stacky curve section} computes the cohomology of $\mathbb G_m$ on
a stacky curve, and Section \ref{mu_r gerbe section} applies this to prove our main theorem
on $\mu_r$-gerbes.
We collect in Appendix \ref{appendix} several results on group cohomology and spectral
sequences which are used throughout the paper.

\subsection{Conventions}\label{conventions}
We work over a fixed algebraically closed field $\mathbbm k$.
All stacks are assumed to be tame,
separated, Deligne-Mumford, and of finite type over $\mathbbm k$.
When $\mathcal X$ is one-dimensional, we will refer to
$\HH^2(\mathcal X,\mathbb G_m)$ as the \textit{Brauer group} of $\mathcal X$.
A Deligne-Mumford stack is \textit{tame} if $\fieldchar\mathbbm k$
does not divide the order of any geometric stabilizer group, and an algebraic stack is \textit{tame} if it
has finite inertia and
all geometric stabilizers are linearly reductive (see \cite{AOV08}).

We take a \textit{stacky curve} to be a separated one-dimensional Deligne-Mumford
stack with trivial generic stabilizer of finite type over a field, though we tend to avoid this terminology
(the eagle-eyed reader who has seen the title of this paper will notice that we failed
this immediately). Instead, we usually
directly state the conditions we are assuming, since there are conflicting definitions
of \textit{stacky curve} in the literature.

We use Lemma \ref{reduction lemma} to assume without loss of generality that all stacks in this paper
are reduced.

All groups in this paper act trivially on $\mathbbm k^{\times}$ unless otherwise stated.

\subsection{Acknowledgments}\label{acknowledgments}
This project was born out of an upcoming paper with William C. Newman.
I am heavily
indebted to him for hours of hard work and insightful comments on draft versions
of this paper. The results of this paper will be used in this upcoming work to both
help understand the structure of $\mgbar{1,1}(B\mu_n)$ as well as to compute
the Picard groups of nodal stacky curves.

I would also like to thank Cherry Ng for helpful conversations on group cohomology,
Jackson Morris for fantastic advice on how to compute using
spectral sequences, and Minseon Shin for suggesting a simplification of an early version
of Proposition \ref{Meier structure}.

\section{Basics on gerbes}\label{gerbe background}

\begin{definition}
A stack $\mathcal X$ over a site $\mathcal S$ is called a \textit{gerbe} if
\begin{enumerate}
\item for every object $T\in\mathcal S$, there exists a covering $\{T_i\rightarrow T\}$ in $\mathcal S$
such that each $\mathcal X(T_i)$ is non-empty, and
\item for objects $x,y\in\mathcal X$ over $T\in\mathcal S$, there exists a covering $\{T_i\rightarrow T\}$
and isomorphisms $x|_{T_i}\rightarrow y|_{T_i}$ for each $i$.
\end{enumerate}
\end{definition}

\begin{definition}
Let $G$ be an abelian sheaf on $\mathcal S$. A stack $\mathcal X$ over $\mathcal S$
is called a \textit{$G$-gerbe} (or a \textit{banded $G$-gerbe}) if it is a gerbe
together with the data of isomorphisms $\psi_x:G|_T\rightarrow\Aut_T(x)$ of sheaves
for each object $x\in\mathcal X(T)$. We require that for each isomorphism $\alpha:x\xrightarrow{\sim}y$
over $T$, the diagram
$$
\begin{tikzcd}
& G|_T \arrow[dl, swap, "\psi_x"] \arrow[dr, "\psi_y"]& \\
\Aut_T(x) \arrow[rr, "\text{Inn}_{\alpha}"] & & \Aut_T(y)
\end{tikzcd}
$$
commutes.
\end{definition}

Loosely speaking, this says that $G$-gerbes are a generalization of the notion of classifying stack,
in the sense that a $G$-gerbe is a stack which \textit{locally} is isomorphic to $BG$, much in the same way
that $G$-torsors are locally isomorphic to $G$. It is a foundational fact that $G$-gerbes
are classified by $\HH^2(\mathcal S, G)$ (\cite[Theorem 12.2.8]{Ols16}).

\begin{proposition}\label{R^1g_*}
Let
$\pi:\mathcal X\rightarrow X$ be a $\mu_r$-gerbe over a Deligne-Mumford stack.
Then there is a natural isomorphism
$$
\mathbf R^1\pi_*\mathbb G_m\cong\Hom(\mu_r,\mathbb G_m)=\mathbb Z/r.
$$
\end{proposition}

\begin{proof}
Observe that to each line bundle $L\in\Pic(\mathcal X)$ we may associate a character
$\chi_L\in\Hom(\mu_r,\mathbb G_m)$. For each \'{e}tale $V\rightarrow X$,
we have the exact sequence
$$
0\rightarrow\Pic(V)\rightarrow\Pic(\mathcal X|_V)\xrightarrow{L\mapsto\chi_L}\Hom(\mu_r,\mathbb G_m).
$$
Since $\mathcal X$ is \'{e}tale locally trivial, this sequence is \'{e}tale locally surjective on the right.
That is, $\Hom(\mu_r,\mathbb G_m)$ is precisely the sheafification of $\Pic(\mathcal X)/\Pic(X)$,
which is $\mathbf R^1\pi_*\mathbb G_m$.
\end{proof}

Our principal gerbes of study will be \textit{root gerbes}, which were introduced in the
literature in \cite{Cad07}.

\begin{definition}
    Let $X$ be a Deligne-Mumford stack. Then the data of a line bundle $L$ on $X$ is equivalent
    to the data of a morphism $X\xrightarrow{[L]}B\mathbb G_m$. Letting $r$ denote the
    $r^{\text{th}}$ power map $r:\mathbb G_m\rightarrow\mathbb G_m$, the \textit{$r^{\text{th}}$
    root gerbe}, written $X(\sqrt[r]{L})$, is defined as the following fiber product
    $$
    \begin{tikzcd}
        X(\sqrt[r]{L})\arrow[r]\arrow[d] &  B\mathbb G_m\arrow[d, "r"]\\
        X\arrow[r, "L"] & B\mathbb G_m.
    \end{tikzcd}
    $$
\end{definition}

A morphism from a scheme $T$ to $X(\sqrt[r]L)$ is equivalent to the data of a morphism
$f:T\rightarrow X$, a line bundle $L'$ on $T$, and an isomorphism $L'^{\otimes r}\cong f^*(L)$.

\begin{proposition}\label{root injectivity}
Let $X$ be an Deligne-Mumford stack.
Then a $\mu_r$-gerbe
$\pi:\mathcal X\rightarrow X$ is a root gerbe
if and only if $\HH^2(X,\mathbb G_m)\rightarrow\HH^2(\mathcal X,\mathbb G_m)$ is injective.
\end{proposition}

\begin{proof}
Consider the Leray spectral sequence
$$
E_2^{p,q}=\HH^p(X,\mathbf R^q\pi_*\mathbb G_m)
\implies\HH^{p+q}(\mathcal X,\mathbb G_m).
$$
Since $\pi_*\mathbb G_m=\mathbb G_m$,
the beginning of the associated sequence of low-degree terms is
$$
0\rightarrow\HH^1(X,\mathbb G_m)\rightarrow\HH^1(\mathcal X,\mathbb G_m)\rightarrow
\HH^0(X,\mathbf R^1\pi_*\mathbb G_m)=$$
$$
\HH^0(X,\Hom(\mu_r,\mathbb G_m))\xrightarrow{d_2}
\HH^2(X,\mathbb G_m)\xrightarrow{\pi^*}\HH^2(\mathcal X,\mathbb G_m).
$$
Therefore the injectivity of $\pi^*$ is equivalent to the triviality of $d_2$.
By \cite[Theorem 1.2]{Lop23}, $d_2$ is just the map
which sends $\epsilon:\mu_r\rightarrow\mathbb G_m$ to the image of
$[\mathcal G]\in\HH^2(X,\mu_r)$ under the pushforward map
$$
\epsilon_*:\HH^2(X,\mu_r)\rightarrow\HH^2(X,\mathbb G_m).
$$
Setting $\epsilon$ to be the canonical inclusion $\mu_r\hookrightarrow\mathbb G_m$,
the Kummer sequence shows that
$d_2$ is trivial if and only if $[\mathcal G]$ is in
the image of $\HH^1(X,\mathbb G_m)$. That is, if and only if
$\mathcal G$ is a root gerbe.
\end{proof}

\begin{proposition}[\cite{AM20}, Proposition 2.5 (iv)]\label{AM}
    If $X$ is a regular and noetherian Deligne-Mumford stack, and if $U\subseteq X$ is
    a dense open subset, then the restriction map $\HH^2(X,\mathbb G_m)\rightarrow \HH^2(U,\mathbb G_m)$
    is injective.
\end{proposition}

\begin{theorem}[Tsen's Theorem]\label{Tsen's Theorem}
Let $X$ be a one-dimensional scheme over an
algebraically closed field $\mathbbm k$. Then
$\HH^i(X,\mathbb G_m)=0$ for $i\geq 2$.
\end{theorem}

\begin{theorem}[Stacky Tsen's Theorem]\label{stacky Tsen's}
Let $\mathcal C$ be a smooth separated one-dimensional
Deligne-Mumford stack with trivial generic stabilizer
of finite type over an algebraically closed field.
Then $\HH^2(\mathcal C,\mathbb G_m)=0$.
\end{theorem}

\begin{proof}
Since there are only finitely many points
which have non-trivial stabilizer, this follows
from Tsen's Theorem and Proposition \ref{AM}.
\end{proof}

\begin{note}
    Note the contrast between Theorem \ref{Tsen's Theorem}, which applies to \textit{all} curves,
    and Theorem
    \ref{stacky Tsen's}, which applies only when $\mathcal C$ is \textit{smooth}.
    Intuitively, this says that for a stacky curve
    (which is generically schemey), any non-triviality
    in $\HH^2(\mathcal C,\mathbb G_m)$ must come from stabilizers at singularities,
    which is precisely what we will see in Proposition \ref{C G_m cohomology} and
    Note \ref{C G_m cohomology note}. In Example \ref{stacky Tsen example} we present
    a minimal example of a stacky curve with non-vanishing $\HH^2$.
\end{note}

\subsection{Some results of Meier}\label{Meier section}
The following results, with small adjustments to fit our
context, appear in \cite{ACV03} and \cite{Mei18}.

\begin{definition}[\cite{ATW20}, 2.3.1]
    A morphism $f:X\rightarrow Y$ of stacks is called a \textit{coarsening}
    if the inertia $\mathcal I_{X/Y}$ is finite over $X$ and for any flat $Z\rightarrow Y$
    from an algebraic space $Z$ the resulting base change
    $X\times_YZ\rightarrow Z$ is a coarse moduli space.
\end{definition}

\begin{proposition}[\cite{ACV03}, Proposition A.0.1 \cite{Mei18}, Proposition 8]\label{Meier structure}
    Let $\mathcal X$ be a separated tame Deligne-Mumford stack of finite type
    over $\mathbbm k$ with coarsening $f:\mathcal X\rightarrow X$.
    Let $x:\Spec K\rightarrow X$ be a geometric point. Let $G=(\mathcal I_{\mathcal X/X})_x$ be
    the relative stabilizer.
    Then for all $n\geq 0$
    $$
    (\mathbf R^nf_*\mathbb G_m)_x=\HH^n(G,R^{\times})
    $$
    for some strictly Henselian local ring $R$ with residue field $K$. Moreover, the group
    $G$ acts trivially on $K$.
\end{proposition}

\begin{proof}
This is originally stated for coarse moduli spaces.
We reduce to that case by observing that since
$\mathcal X\rightarrow X$ is a coarsening we may restrict to any
\'{e}tale neighborhood from a scheme $V\rightarrow X$
and then be in the setting of a coarse moduli space
$\mathcal X_V\rightarrow V$.

This now essentially follows from \cite[Proposition A.0.1]{ACV03}, which provides an equality
$$
(\mathbf R^nf_*\mathbb G_m)_x=\HH^n(G,(\mathbb G_m)_{x'}),
$$
where $x'$ is the preimage of $x$ in $\mathcal X$.
All that's left to do is observe the following: $(\mathbb G_m)_{x'}=R^{\times}$ where $R$
is a strictly Henselian local ring; the residue field of $R$ is $K$; and
since the relative geometric stabilizer at $X$ is $G$ we must have that $G$ acts
trivially on $K$.
\end{proof}

\begin{corollary}[\cite{Mei18}, Lemma 9]\label{Meier stalk}
Let $\mathcal X$ and $G$ be as in Proposition \ref{Meier structure}. Then for all geometric points
$x:\Spec K\rightarrow X$ and for all $n\geq 2$,
$$
(\mathbf R^nf_*\mathbb G_m)_x=\HH^n(G,K^{\times}).
$$
\end{corollary}

\begin{note}
This is originally shown by Meier only for the case $n=2$ and only for the $\ell$-torsion with
$\ell\neq\fieldchar K$. However, the argument actually works in our stated generality since
$\HH^n(G,-)$ is $|G|$-torsion and $|G|$ is coprime to $\fieldchar K$ since
$\mathcal X$ is tame and Deligne-Mumford.
For the sake of completion, we reproduce the argument here.
\end{note}

\begin{proof}[Proof of Corollary \ref{Meier stalk}]
Let $p=\fieldchar K$.
Using Proposition \ref{Meier structure}, all that's left to show is that $\HH^n(G,R^{\times})=
\HH^n(G, K^{\times})$.
Consider the exact sequence
$$
1\rightarrow\ker\rightarrow R^{\times}[1/p]
\rightarrow K^{\times}[1/p]\rightarrow 1,
$$
where if $p=0$ we take adjoining $1/p$ to be trivial. Since $R$ and $K$
are both strictly Henselian, we have that  $R^{\times}[1/p]$ and
$K^{\times}[1/p]$ are both divisible. Therefore $\ker$ is divisible.
Now by \cite[Tag 06RR]{Sta},
the torsion of $R^{\times}[1/p]$ maps isomorphically onto the torsion of
$K^{\times}[1/p]$, and hence $\ker$ is torsion-free. The structure theorem
for divisible groups (see \cite[Theorem 23.1]{Fuc70}) then implies that $\ker$ is a $\mathbb Q$-vector space
and hence has vanishing higher cohomology.

Therefore, taking cohomology, we see that
$$
\HH^n(G,R^{\times})=
\HH^n(G,R^{\times}[1/p])=
\HH^n(G,K^{\times}[1/p])=
\HH^n(G,K^{\times})
$$
for all $n\geq 2$, where the first and last equalities use the fact that $|G|$ is coprime to
$p$.
\end{proof}

We add here an additional result inspired by the above.

\begin{proposition}\label{Meier mu_r}
Let $\mathcal X$ and $G$ be as in Proposition \ref{Meier structure}. Then for all
geometric points $x:\Spec K\rightarrow X$, all $n\geq1$, and all $r$ coprime to $\fieldchar K$,
$$
(\mathbf R^nf_*\mu_r)_x=\HH^n(G,\mu_r(K)).
$$
\end{proposition}

\begin{proof}
This reduces to showing that if $\mathcal X=[\Spec R/G]$,
then $\HH^n(\mathcal X,\mu_r)=\HH^n(G,\mu_r(K))$. By the descent spectral sequence we have
$$
\HH^n(\mathcal X,\mu_r)=\HH^n(G,\mu_r(R))=\HH^n(G,\mu_r(K)),
$$
where the last equality comes from $R$ being a strictly Henselian local ring.
\end{proof}

\section{The cohomology of a stacky curve}\label{stacky curve section}
In this section and the next, we fix the following data:
\begin{itemize}
\item
$\mathcal C$ is a tame one-dimensional separated Deligne-Mumford stack with trivial
generic stabilizer of finite type
over an algebraically closed field $\mathbbm k$,
\item $\pi:\mathcal G\rightarrow\mathcal C$
is a $\mu_r$-gerbe,
\item $g:\mathcal G\rightarrow C$ is the coarse moduli space (note that
$C$ is the coarse moduli space of both $\mathcal G$ and $\mathcal C$),
\item $p_1,\dots,p_n$ are the points of $C$ whose preimage in $\mathcal C$
have non-trivial stabilizers $G_i$,
\item the preimages of $p_1,\dots,p_n$ in $\mathcal G$ have stabilizers $E_1,\dots,E_n$,
so that for each $i$ we have the exact sequence of groups
$$
0\rightarrow \mu_r\rightarrow E_i\rightarrow G_i\rightarrow 0,
$$
\item and the trivial action of $G_i$ and $E_i$ on $\mathbbm k^{\times}$ (see Proposition
\ref{Meier structure}).
\end{itemize}
This can be summarized in the diagram
$$
\begin{tikzcd}
    \mathcal G\arrow[d,"\pi"]\arrow[dd, bend right, swap, "g"]\\
    \mathcal C\arrow[d,"f"]\\
    C.
\end{tikzcd}
$$

We open by making the following observation.

\begin{lemma}\label{reduction lemma}
Let $\mathcal X$ be a separated tame one-dimensional Deligne-Mumford stack
of finite type over a (not necessarily algebraically closed field) $\mathbbm k$.
If $\dim X\leq 1$, then
we have
$$
\HH^2(\mathcal X,\mathbb G_m)=\HH^2(\mathcal X_{red},\mathbb G_m).
$$
\end{lemma}

\begin{proof}
Let $I$ be the sheaf of ideals of $\mathcal X_{red}$. Since $\mathcal X$ is finite type
over $\mathbbm k$, we may factor $i:\mathcal X_{red}\rightarrow\mathcal X$ through a sequence of
square-zero thickenings and hence may assume that $I$ is itself square-zero.

We then obtain an exact sequence
$$
1\rightarrow 1+I\rightarrow\mathbb G_m\rightarrow i_*\mathbb G_m\rightarrow 1
$$
with associated long exact sequence
$$
\HH^2(\mathcal X,I)\rightarrow\HH^2(\mathcal X,\mathbb G_m)\rightarrow
\HH^2(\mathcal X,i_*\mathbb G_m)\rightarrow\HH^3(\mathcal X,I).
$$
We claim that $\HH^2(\mathcal X,I)=\HH^3(\mathcal X,I)=0$, which would finish the proof
as we would then have
$$
\HH^2(\mathcal X,\mathbb G_m)=\HH^2(\mathcal X,i_*\mathbb G_m)=
\HH^2(\mathcal X_{red},\mathbb G_m).
$$

To prove the claim,
since $\mathcal X$ is tame, the pushforward $\pi_*:\text{QCoh}(\mathcal X)\rightarrow
\text{QCoh(X)}$ is exact, and hence it suffices to check that
$\HH^2(X,\pi_*I)=\HH^3(X,\pi_*I)=0$, which holds since \'{e}tale cohomology agrees
with Zariski cohomology for quasi-coherent sheaves and since all Zariski cohomology past $\HH^1$ vanishes
on $X$ by the Grothendieck vanishing theorem.
\end{proof}

\begin{note}
We use Lemma \ref{reduction lemma}
to without loss of generality treat our stacks as if they were reduced, though we will
occasionally directly translate statements into the non-reduced context.
\end{note}

Consider the Leray spectral sequence for the sheaf $\mathbb G_{m,\mathcal G}$ and
the map $\pi:\mathcal G\rightarrow\mathcal C$:
$$
E_2^{p,q}=\HH^p(\mathcal C,\mathbf R^q\pi_*\mathbb G_{m,\mathcal G})\implies
\HH^{p+q}(\mathcal G,\mathbb G_{m,\mathcal G}).
$$
Using the fact that $\pi_*\mathbb G_{m,\mathcal G}=\mathbb G_{m,\mathcal C}$,
the associated sequence of low-degree terms is

\begin{multline}\label{G to C low degree initial}
0\rightarrow\HH^1(\mathcal C,\mathbb G_m)\rightarrow\HH^1(\mathcal G,\mathbb G_m)\rightarrow
\HH^0(\mathcal C,\mathbf R^1\pi_*\mathbb G_m)\rightarrow\HH^2(\mathcal C,\mathbb G_m)\rightarrow\\
\ker\left[\HH^2(\mathcal G,\mathbb G_m)\rightarrow\HH^0(\mathcal C,\mathbf R^2\pi_*\mathbb G_m)\right]
\rightarrow
\HH^1(\mathcal C,\mathbf R^1\pi_*\mathbb G_m)\rightarrow\HH^3(\mathcal C,\mathbb G_m).
\end{multline}

Therefore, before we proceed to our computation of $\HH^2(\mathcal G,\mathbb G_m)$,
we must first understand the cohomology of $\mathbb G_m$ on $\mathcal C$, as well
as $\mathbf R^2\pi_*\mathbb G_m$.
We open our analysis by using the Grothendieck spectral sequence to relate
$\mathbf R^*f_*\mathbb G_m$, $\mathbf R^*g_*\mathbb G_m$, and $\mathbf R^*\pi_*\mathbb G_m$.

\subsection{Grothendieck spectral sequence}
To understand our various higher pushforwards, we consider the Grothendieck spectral sequence
for the composition $g=f\circ\pi$ and the sheaf $\mathbb G_m$:
$$
E_2^{pq}=\left(\mathbf R^pf_*\circ\mathbf R^q\pi_*\right)(\mathbb G_m)\implies
\mathbf R^{p+q}g_*(\mathbb G_m).
$$
Accounting for the fact that the pushforward of $\mathbb G_m$ is $\mathbb G_m$,
the beginning of the sequence of low-degree terms is
$$
0\rightarrow\mathbf R^1f_*\mathbb G_m\xrightarrow{\alpha}\mathbf R^1g_*\mathbb G_m
\xrightarrow{\beta} f_*(\mathbf R^1\pi_*\mathbb G_m)\xrightarrow{\gamma}
\mathbf R^2f_*\mathbb G_m\xrightarrow{\delta}\mathbf R^2g_*\mathbb G_m.
$$
We may extract two short exact sequences from this:
$$
0\rightarrow\mathbf R^1f_*\mathbb G_m\rightarrow\mathbf R^1g_*\mathbb G_m
\rightarrow \ker\gamma\rightarrow 0
$$
and
$$
0\rightarrow \ker\gamma \rightarrow f_*(\mathbf R^1\pi_*\mathbb G_m)
\rightarrow \ker\delta\rightarrow 0.
$$

\begin{lemma}\label{higher cohomology}
For all $i\geq 2$,
$$
\HH^i(C,\mathbf R^1g_*\mathbb G_m)\cong\HH^i(C,f_*(\mathbf R^1\pi_*\mathbb G_m)).
$$
The isomorphism holds for $i=1$ if $\mathbf R^2f_*\mathbb G_m\rightarrow
\mathbf R^2g_*\mathbb G_m$ is injective.
\end{lemma}

\begin{proof}
Notice that since $\mathcal C$ has trivial generic stabilizer, Corollary \ref{Meier stalk} says that
for $i\geq 1$ the sheaves
$\mathbf R^if_*\mathbb G_m$
are supported on a finite set (the points where $\mathcal C$ has non-trivial stack structure), and so
their higher cohomology vanishes. Hence for $i\geq 1$
$$
0\rightarrow\HH^i(C,\mathbf R^1g_*\mathbb G_m)\rightarrow
\HH^i(C,\ker\gamma)\rightarrow0.
$$
That is,
$$
\HH^i(C,\mathbf R^1g_*\mathbb G_m)\cong\HH^i(C,\ker\gamma).
$$
But using the second exact sequence from above gives, in degree 0 and 1,
$$
\HH^0(C,f_*(\mathbf R^1\pi_*\mathbb G_m))\rightarrow\HH^0(C,\ker\delta)\rightarrow
\HH^1(C,\ker\gamma)\rightarrow\HH^1(C,f_*(\mathbf R^1\pi_*\mathbb G_m))\rightarrow0
$$
and, in degree above 1,
$$
0\rightarrow\HH^{i}(C,\ker\gamma)\rightarrow\HH^{i}(C,
f_*(\mathbf R^1\pi_*\mathbb G_m))\rightarrow0.
$$
This establishes that for all $i\geq 2$ we have
$$
\HH^i(C,\mathbf R^1g_*\mathbb G_m)\cong\HH^i(C,\ker\gamma)\cong
\HH^i(C,f_*(\mathbf R^1\pi_*\mathbb G_m))
$$
and that the isomorphism holds for $i=1$ if
$\HH^0(C,\ker\delta)=0$, i.e. if $\mathbf R^2f_*\mathbb G_m\rightarrow
\mathbf R^2g_*\mathbb G_m$
is injective on global sections. In particular this holds if the map of sheaves is injective.
\end{proof}

\begin{note}\label{R^2 injectivity note}
We will see in Corollary \ref{R^2 injectivity} that $\mathbf R^2f_*\mathbb G_m\rightarrow
\mathbf R^2g_*\mathbb G_m$ is injective precisely
when $\mathcal G$ is a root gerbe over $\mathcal C$, and that in fact the map on sheaves
is determined by the map on global sections.
\end{note}

\begin{lemma}\label{R^2k}
For all $k\geq1$,
$$
\mathbf R^{2k}g_*\mathbb G_m=\bigoplus_{i=1}^n p_{i_*}\HH^{2k}(E_i,\mathbbm k^{\times}).
$$
\end{lemma}

\begin{proof}
By Corollary \ref{Meier stalk} and Proposition \ref{cyclic cohomology}, for all $x\neq p_1,\dots,p_n$ we have
$(\mathbf R^{2k}g_*\mathbb G_m)_x=\HH^{2k}(\mu_r,\mathbbm k^{\times})=0$.
Therefore $\mathbf R^{2k}g_*\mathbb G_m$ is supported at the $p_i$'s and hence
$$
\mathbf R^{2k}g_*\mathbb G_m=\bigoplus_{i=1}^n(\mathbf R^{2k}g_*\mathbb G_m)_{p_i}=
\bigoplus_{i=1}^n p_{i_*}\HH^{2k}(E_i,\mathbbm k^{\times}),
$$
by Corollary \ref{Meier stalk}.
\end{proof}

\subsection{The cohomology of $\mathbb G_m$ on $\mathcal C$}

\begin{lemma}\label{R^kf}
For all $k\geq 2$,
$$
\mathbf R^kf_*\mathbb G_m=\bigoplus_{i=1}^n p_{i_*}\HH^k(G_i,\mathbbm k^{\times}).
$$
\end{lemma}

\begin{proof}
Let $x$ be a geometric point of $C$ and $G_x$ the stabilizer of the preimage of $x$ in $\mathcal C$.
By Corollary \ref{Meier stalk}, $(\mathbf R^kf_*\mathbb G_m)_x=\HH^k(G_x,\mathbbm k^{\times})$,
and this vanishes away from $p_1,\dots,p_n$ since $G_x=0$ generically.
Hence $\mathbf R^kf_*\mathbb G_m$ is supported
at $p_1,\dots,p_n$, and the lemma follows.
\end{proof}

\begin{proposition}\label{C G_m cohomology}
For all $k\geq 2$
we have that
$$
\HH^k(\mathcal C,\mathbb G_m)=\HH^0(C,\mathbf R^kf_*\mathbb G_m)=
\bigoplus_{i=1}^n\HH^k(G_i,\mathbbm k^{\times}),
$$
where, for even $k$, this is concentrated at the singular points of $\mathcal C_{red}$.
\end{proposition}

\begin{proof}
We use the Leray spectral sequence
$$
E_2^{p,q}=\HH^p(C,\mathbf R^qf_*\mathbb G_m)\implies\HH^{p+q}(\mathcal C,\mathbb G_m).
$$
The terms $E_2^{p,q}$ for $(p,q)\geq(1,1)$ all vanish, since the sheaves
$\mathbf R^qf_*\mathbb G_m$ are supported on finite sets (by Corollary \ref{Meier stalk})
and hence have vanishing
higher cohomology. Additionally, the terms vanish for $(p,0)$ when
$p\geq 2$ by Theorem \ref{Tsen's Theorem}.
Therefore for $k\geq 2$ the only non-zero term in the filtration
$$
0\subseteq H^kF^k\subseteq\dots\subseteq H^0F^k=\HH^k(\mathcal C,\mathbb G_m)
$$
is the term $E_2^{0,k}$, and so by Lemma \ref{R^kf}
$$
\HH^k(\mathcal C,\mathbb G_m)=E_2^{0,k}=\HH^0(C,\mathbf R^kf_*\mathbb G_m)=
\bigoplus_{i=1}^n\HH^k(G_i,\mathbbm k^{\times}).
$$

The last statement follows since the only possible stabilizers
at smooth points are cyclic (this is using that $\mathcal C$ has trivial
generic stabilizer), and cyclic groups have vanishing even cohomology
by Proposition \ref{cyclic cohomology}.
\end{proof}

\begin{note}\label{C G_m cohomology note}
We therefore get a characterization of when restriction to an open
$\HH^2(\mathcal C,\mathbb G_m)\rightarrow
\HH^2(\mathcal U,\mathbb G_m)$ fails to be injective:
injectivity fails
precisely when $\mathcal C\setminus\mathcal U$ contains a singular point of $\mathcal C_{red}$
whose stabilizer has non-trivial second cohomology. This explains where the regularity condition
in Proposition \ref{AM} comes from, and shows that it is truly necessary.
\end{note}

\begin{example}\label{stacky Tsen example}
Let $G=\mu_n\rtimes\mathbb Z/2$ and set
$$
\mathcal C=\left[\frac{\mathbbm k[x,y]/(xy)}{G}\right],
$$
where $\mu_n$ acts as $t\cdot(x,y)=(tx,t^ay)$ such that $a^2\equiv1\pmod n$, and where
$\mathbb Z/2$ acts by interchanging $x$ and $y$. Then $\mathcal C$ has a node with a
$BG$ at the origin, and its coarse moduli space is $\mathbb A^1$.
A somewhat involved group cohomology computation in an upcoming work with Newman
will show that, for certain values of $a$ and $n$, $\HH^2(G,\mathbbm k^{\times})\neq0$.
This gives, in a sense, the simplest possible example of a (generically schemey)
stacky curve with non-trivial Brauer group, since a node is the simplest singularity type
which supports non-cyclic stabilizers.
\end{example}

\section{The Brauer group of a $\mu_r$-gerbe}\label{mu_r gerbe section}
In this section we will complete our computation of the Brauer group of a $\mu_r$-gerbe
over $\mathcal C$. Since such a gerbe corresponds to a class in $\HH^2(\mathcal C,\mu_r)$, we begin with
one final cohomology computation on $\mathcal C$ to aide in our classification
of root gerbes. This will be one of the crucial ingredients in Theorem \ref{main Brauer statement}.

\subsection{Classifying $\mu_r$-gerbes over $\mathcal C$}

\begin{proposition}\label{C mu_r cohomology}
The second cohomology of $\mu_r$ fits into the following sequence
$$
\HH^2(C,\mu_r)\rightarrow\HH^2(\mathcal C,\mu_r)\rightarrow\bigoplus_{i=1}^n
\HH^2(G_i,\mu_r)\rightarrow0.
$$
That is, a $\mu_r$-gerbe over $\mathcal C$ is completely determined by the choice
of a $\mu_r$-gerbe on $C$ and a $\mu_r$-gerbe at each non-trivial stabilizer of $\mathcal C$.
\end{proposition}

\begin{proof}
Consider the Leray spectral sequence
$$
E_2^{p,q}=\HH^p(C,\mathbf R^qf_*\mu_r)\implies\HH^{p+q}(\mathcal C,\mu_r).
$$
We use Proposition \ref{Meier mu_r} to identify the stalks of $\mathbf R^qf_*\mu_r$
with the group cohomology of $\mu_r$. Then for $q\geq 1$ the sheaf
$\mathbf R^qf_*\mu_r$ is supported at finitely many points and hence has vanishing
higher cohomology.
We then get
$$
E_3^{3,0}=\HH^3(C,\mu_r)=0
$$
by the Kummer sequence,
$$
E_2^{2,1}=\HH^2(C,\mathbf R^1f_*\mu_r)=0,\text{ and}
$$
$$
E_2^{1,1}=\HH^1(C,\mathbf R^1f_*\mu_r)=0.
$$
Then the relevant terms in the filtration on $\HH^2$ are:
$$
\begin{aligned}
    E_{\infty}^{2,0}&=E_3^{2,0}=\frac{\HH^2(C,\mu_r)}
    {\im\left(E_2^{0,1}\rightarrow\HH^2(C,\mu_r)\right)}\\
    E_{\infty}^{1,1}&=0\\
    E_{\infty}^{0,2}&=\ker(E_3^{0,2}\rightarrow E_3^{3,0}=0)\\
    &=E_3^{0,2}=\ker(E_2^{0,2}\rightarrow E_2^{2,1}=0)\\
    &=\HH^0(C,\mathbf R^2f_*\mu_r).
\end{aligned}
$$

Therefore $\HH^2(\mathcal C,\mu_r)$ surjects onto
$\HH^0(C,\mathbf R^2f_*\mu_r)=\bigoplus_{i=1}^n\HH^2(G_i,\mu_r)$
with kernel equal to a quotient of $\HH^2(C,\mu_r)$, as desired.
\end{proof}

\begin{corollary}\label{root fiber}
A $\mu_r$-gerbe on $\mathcal C$ is a root gerbe if and only if it is a root gerbe restricted
to each $BG_i$.
\end{corollary}

\begin{proof}
    Applying the Kummer sequence simultaneously to the three components of
    the exact sequence in Proposition \ref{C mu_r cohomology} yields
    $$
    \begin{tikzcd}
        \HH^2(C,\mu_r)\arrow[r]\arrow[d] & \HH^2(\mathcal C,\mu_r)
        \arrow[r]\arrow[d] & \displaystyle\bigoplus_{i=1}^n\HH^2(G_i,\mu_r)\arrow[d]\\
        \HH^2(C,\mathbb G_m)=0\arrow[r] & \HH^2(\mathcal C,\mathbb G_m)
        \arrow[r,"\cong"] & \displaystyle\bigoplus_{i=1}^n\HH^2(G_i,\mathbbm k^{\times}),
    \end{tikzcd}
    $$
    where the bottom right isomorphism comes from Proposition \ref{C G_m cohomology}.
    A $\mu_r$-gerbe on $\mathcal C$ is a root gerbe if and only if its image in
    $\HH^2(\mathcal C,\mathbb G_m)$ is zero, which happens if and only if, therefore,
    its image in $\HH^2(G_i,\mathbbm k^{\times})$ is zero. That is, if and only
    if it is a root gerbe on each $BG_i$.
\end{proof}

\begin{corollary}\label{R^2 injectivity}
The $\mu_r$-gerbe $\pi:\mathcal G\rightarrow\mathcal C$ is a root gerbe if and only if the natural map
$\mathbf R^2f_*\mathbb G_m\rightarrow\mathbf R^2g_*\mathbb G_m$ is injective.
\end{corollary}

\begin{proof}
By Corollary \ref{root fiber} $\mathcal G$ is a root gerbe if and only if it is a root gerbe
on each fiber. However, by Proposition \ref{root injectivity} this happens if and only if
$$
\HH^2(BG_i,\mathbb G_m)=\HH^2(G_i,\mathbbm k^{\times})\hookrightarrow
\HH^2(E_i,\mathbbm k^{\times})=\HH^2(BE_i,\mathbb G_m)
$$
for all $i$.
By the explicit descriptions of $\mathbf R^2f_*\mathbb G_m$ and
$\mathbf R^2g_*\mathbb G_m$ given in Lemmas \ref{R^2k} and \ref{R^kf}, this is
equivalent to $\mathbf R^2f_*\mathbb G_m\hookrightarrow
\mathbf R^2g_*\mathbb G_m$.
\end{proof}

\begin{lemma}\label{R^3 injection}
The natural morphism $\mathbf R^3f_*\mathbb G_m\rightarrow\mathbf R^3g_*\mathbb G_m$
is injective
if and only if $\HH^3(G_i,\mathbbm k^{\times})\xrightarrow{\inf}\HH^3(E_i,\mathbbm k^{\times})$
is injective for all $i$.
\end{lemma}

\begin{proof}
It is sufficient to check on geometric fibers. The sheaf $\mathbf R^3f_*\mathbb G_m$ vanishes
away from $p_1,\dots,p_n$ by Lemma \ref{R^kf}, and so it is sufficient to check
just at the $p_i$. By Corollary \ref{Meier stalk}, at each $p_i$ we have
$(\mathbf R^3f_*\mathbb G_m)_{p_i}=\HH^3(G_i,\mathbbm k^{\times})$
and $(\mathbf R^3g_*\mathbb G_m)_{p_i}=\HH^3(E_i,\mathbbm k^{\times})$. Therefore
$\mathbf R^3f_*\mathbb G_m\rightarrow\mathbf R^3g_*\mathbb G_m$ is injective
if and only if it is injective on all stalks if and only if the inflation map is injective.
\end{proof}

\subsection{Leray for $g:\mathcal G\rightarrow C$}
The terms of the Leray sequence are
$$
E_2^{p,q}=\HH^p(C,\mathbf R^qg_*\mathbb G_m).
$$

\begin{lemma}\label{E_2^{p,2}}
Whenever $p\geq 1$ we have $E_2^{p,2}=\HH^p(C,\mathbf R^2g_*\mathbb G_m)=0$.
\end{lemma}

\begin{proof}
By Proposition \ref{R^2k}, $\mathbf R^2g_*\mathbb G_m$ is supported on finitely many points
(the $p_i$'s), and so its higher cohomology vanishes.
\end{proof}

\begin{lemma}\label{E_2^{p,1}}
Whenever $p\geq 2$, we have $E_2^{p,1}=\HH^p(C,\mathbb Z/r)$, and this vanishes for
$p\geq 3$.
\end{lemma}

\begin{proof}
By Lemma \ref{higher cohomology} and Proposition \ref{R^1g_*}, we have
$$
\HH^p(C,\mathbf R^1g_*\mathbb G_m)=\HH^p(C,f_*\mathbf R^1\pi_*\mathbb G_m)=
\HH^p(C,f_*(\mathbb Z/r))=\HH^p(C,\mathbb Z/r)
$$
for all $p\geq 2$. We observe here that by Lemma \ref{higher cohomology} and Note \ref{R^2 injectivity note}
that in fact this holds for all $p\geq1$ when $\mathcal G\rightarrow\mathcal C$ is a root gerbe.

By Theorem \ref{Tsen's Theorem}
the Kummer sequence for $C$ at $p\geq 3$ reads
$$
0=\HH^{p-1}(C,\mathbb G_m)\rightarrow\HH^p(C,\mathbb Z/r)\rightarrow
\HH^p(C,\mathbb G_m)=0,
$$
which shows the vanishing for $p\geq 3$.
\end{proof}

Putting all of this together, we get that the
$E_2$ page of our spectral
sequence (with only the relevant terms filled in) looks like
$$
\begin{tikzcd}
\HH^0(C,\mathbf R^3g_*\mathbb G_m) &  E_2^{1,3} &
E_2^{2,3} & E_2^{3,3} & E_2^{4,3}\\
E_2^{0,2} & 0 & 0 & 0 & 0\\
E_2^{0,1} & E_2^{1,1} & \HH^2(C,\mathbb Z/r) & 0 & 0\\
E_2^{0,0} & E_2^{1,0} & 0 & 0 & 0
\end{tikzcd}
$$

\begin{proposition}\label{H^3 injectivity}
The pullback map
$$
\HH^3(\mathcal C,\mathbb G_m)\xrightarrow{\pi^*}\HH^3(\mathcal G,\mathbb G_m)
$$
is injective if and only if
the inflation map
$$
\HH^3(G_i,\mathbbm k^{\times})\xrightarrow{\inf}\HH^3(E_i,\mathbbm k^{\times})
$$
is injective for all $i$.
\end{proposition}

\begin{proof}
The associated filtration on $\HH^3(\mathcal G,\mathbb G_m)$ is
$$
0\subseteq F^3H^3\subseteq F^2H^3\subseteq F^1H^3\subseteq F^0H^3=\HH^3(\mathcal G,\mathbb G_m).
$$
The last term, $F^0H^3$, surjects onto $E_{\infty}^{0,3}$ with kernel $F^1H^3$, so we will compute
$E_{\infty}^{0,3}$. The relevant differentials are
$$
d_2^{0,3}:E_2^{0,3}\rightarrow E_2^{2,2},
$$
$$
d_3^{0,3}:E_3^{0,3}\rightarrow E_3^{3,1},
$$
$$
d_4^{0,3}:E_4^{0,3}\rightarrow E_4^{4,0}.
$$
Looking at the spectral sequence above this proposition, we see that the only non-zero term
appearing in our above differentials is $E_2^{0,3}$, and hence
$$
\HH^3(\mathcal G,\mathbb G_m)\twoheadrightarrow E_{\infty}^{0,3}=E_2^{0,3}=
\HH^0(C,\mathbf R^3g_*\mathbb G_m).
$$

Using Proposition \ref{C G_m cohomology} to rewrite $\HH^3(\mathcal C,\mathbb G_m)=
\HH^0(C,\mathbf R^3f_*\mathbb G_m)$, we see that the
map $\HH^3(\mathcal C,\mathbb G_m)\rightarrow\HH^3(\mathcal G,\mathbb G_m)$ is induced 
by the map $\mathbf R^3f_*\mathbb G_m\rightarrow\mathbf R^3g_*\mathbb G_m$. By Lemma
\ref{R^3 injection}, injectivity of this map is equivalent to injectivity of
the inflation map for each $i$.

\end{proof}

\subsection{Leray for $\pi:\mathcal G\rightarrow\mathcal C$}\label{Leray for pi}
We now revisit the Leray spectral sequence for $\pi$ and $\mathbb G_m$. This will lead
to an exact sequence containing $\HH^2(\mathcal G,\mathbb G_m)$, and we give multiple
conditions which simplify its behavior.

\begin{proposition}\label{G to C low degree proposition}
The sequence of low-degree terms can be extended to
\begin{multline}\label{G to C low degree}
0\rightarrow\HH^1(\mathcal C,\mathbb G_m)\rightarrow\HH^1(\mathcal G,\mathbb G_m)\rightarrow
\HH^0(\mathcal C,\mathbb Z/r)\rightarrow\HH^2(\mathcal C,\mathbb G_m)\rightarrow\\
\HH^2(\mathcal G,\mathbb G_m)\rightarrow
\HH^1(\mathcal C,\mathbb Z/r)\rightarrow\HH^3(\mathcal C,\mathbb G_m)\rightarrow\
\HH^3(\mathcal G,\mathbb G_m).
\end{multline}
\end{proposition}

\begin{proof}
This is just the standard exact sequence of low-degree terms (Sequence \ref{G to C low degree initial}),
where we add the following observations.

First, we have $\mathbf R^2\pi_*\mathbb G_m=0$ since for all geometric points $x$
of $\mathcal C$ we have $(\mathbf R^2\pi_*\mathbb G_m)_x=\HH^2(\mu_r,\mathbbm k^{\times})=0$
by Corollary \ref{Meier stalk} and Proposition \ref{cyclic cohomology}.
Therefore
$$
\ker\left[\HH^2(\mathcal G,\mathbb G_m)\rightarrow\HH^0(\mathcal C,\mathbf R^2\pi_*
\mathbb G_m)\right]=\HH^2(\mathcal G,\mathbb G_m).
$$
Additionally, by Lemma \ref{low-degree}
we may extend the sequence one additional term to $\HH^3(\mathcal G,\mathbb G_m)$,
since $E_2^{0,2}=\HH^0(\mathcal C,\mathbf R^2\pi_*\mathbb G_m)=0$.
\end{proof}

We would like to extract a short exact sequence containing $\HH^2(\mathcal G,\mathbb G_m)$
from this.
We've shown in Lemma \ref{root injectivity} and Corollary \ref{root fiber}
that $\HH^2(\mathcal C,\mathbb G_m)\rightarrow
\HH^2(\mathcal G,\mathbb G_m)$ is injective if and only if $\mathcal G$ is a root gerbe over $\mathcal C$
if and only if $BE_i\rightarrow BG_i$ is a root gerbe on every fiber.
By Proposition \ref{H^3 injectivity}, we have surjectivity of $\HH^2(\mathcal G,\mathbb G_m)
\rightarrow\HH^1(\mathcal C,\mathbb Z/r)$ if and only if $\HH^3(\mathcal C,\mathbb G_m)\rightarrow
\HH^3(\mathcal G,\mathbb G_m)$ is injective
if and only if the inflation map
$\HH^3(G_i,\mathbbm k^{\times})\rightarrow\HH^3(E_i,\mathbbm k^{\times})$ is injective
for all $i$.
This, combined with Proposition \ref{C G_m cohomology} to rewrite
$$
\HH^2(\mathcal C,\mathbb G_m)=\bigoplus_{i=1}^n\HH^2(G_i,\mathbbm k^{\times}),
$$
yields the first three statements of the following.

\begin{theorem}\label{main Brauer statement}
Suppose that $\mathcal C$ is a tame separated one-dimensional Deligne-Mumford stack
with trivial generic stabilizer
of finite type over an
algebraically closed field $\mathbbm k$. Let
$f:\mathcal C\rightarrow C$ be the coarse moduli space of $\mathcal C$.
Suppose that $p_1,\dots,p_n$ are the geometric points
in $C$ whose preimage in $\mathcal C$ have non-trivial stabilizers $G_1,\dots,G_n$.
Let $\pi:\mathcal G\rightarrow\mathcal C$ be a
$\mu_r$-gerbe, and denote the stabilizers in $\mathcal G$ over $p_1,\dots,p_n$ as $E_1,\dots,E_n$.
Let $G_i$ and $E_i$ act trivially on $\mathbbm k^{\times}$ for each $i$.
Then the following sequence is exact
$$
\bigoplus_{i=1}^n\HH^2(G_i,\mathbbm k^{\times})\rightarrow
\HH^2(\mathcal G,\mathbb G_m)\rightarrow
\HH^1(\mathcal C,\mathbb Z/r).
$$
Moreover,
\begin{enumerate}
\item the left map is injective if and only if $\mathcal G$ is a root gerbe over $\mathcal C$,
\item $\mathcal G$ is a root gerbe over $\mathcal C$ if and only if it is a root gerbe on each fiber
$BE_i\rightarrow BG_i$,
\item the right map is surjective if and only if the inflation map
$$
\inf:\HH^3(G_i,\mathbbm k^{\times})\rightarrow\HH^3(E_i,\mathbbm k^{\times})
$$
is injective for all $i=1,\dots,n$,
\item if $\gcd(r,\left|G_i\right|)=1$ for all $i=1,\dots,n$, then:
\begin{enumerate}
\item $\pi:\mathcal G\rightarrow\mathcal C$
is automatically a root gerbe,
\item the inflation map is automatically injective, and
\item we have the following splitting
$$
\HH^2(\mathcal G,\mathbb G_m)=\HH^1(C,\mathbb Z/r)\oplus\bigoplus_{i=1}^n\HH^2(G_i,\mathbbm k^{\times}),
$$
\end{enumerate}
\item more generally,
in the situation that the sequence is right-exact, to show that the sequence splits it suffices
to show that on each fiber $BE_i\rightarrow BG_i$ there is a section
$$
\begin{tikzcd}
\HH^2(G_i,\mathbbm k^{\times})\arrow[r, "\inf"] & \HH^2(E_i,\mathbbm k^{\times})\arrow[l, dashed,
bend right].
\end{tikzcd}
$$
\end{enumerate}
\end{theorem}

\begin{proof}
All that remains to be shown are statements (4) and (5).

For (4),
The Kummer sequence on $\mathcal C$ reads
$$
\HH^1(\mathcal C,\mathbb G_m)\rightarrow\HH^2(\mathcal C,\mu_r)\rightarrow
\HH^2(\mathcal C,\mathbb G_m)\xrightarrow{\cdot r}\HH^2(\mathcal C,\mathbb G_m),
$$
and so the $\mu_r$-gerbe $\mathcal G$ on $\mathcal C$ (corresponding to the class
$[\mathcal G]\in\HH^2(\mathcal C,\mu_r)$)
is a root gerbe if and only if its image in
$\HH^2(\mathcal C,\mathbb G_m)$ is trivial. In Proposition \ref{C G_m cohomology} we have computed that
$$
\HH^2(\mathcal C,\mathbb G_m)=\bigoplus_{i=1}^n\HH^2(G_i,\mathbbm k^{\times}),
$$
and all elements in this group are $\prod|G_i|$-torsion by Proposition \ref{group cohomology torsion}.
But this is coprime to $r$, and all elements of $\HH^2(\mathcal C,\mu_r)$ are $r$-torsion. Therefore
the image of any element of $\HH^2(\mathcal C,\mu_r)$ in $\HH^2(\mathcal C,\mathbb G_m)$
must vanish, and so
$\mathcal G$ is a root gerbe. Thus the sequence
is injective on the left.

To get surjectivity on the right side of our sequence, it suffices to show that the inflation map
$\HH^3(G_i,\mathbbm k^{\times})\rightarrow\HH^3(E_i,\mathbbm k^{\times})$ is injective
for all $i$.
So consider the Lyndon-Hochschild-Serre spectral sequence for
$$
0\rightarrow\mu_r\rightarrow E_i\rightarrow G_i\rightarrow0.
$$
Since $E_2^{0,2}=\HH^0(G_i,\HH^2(\mu_r,\mathbbm k^{\times}))=0$ (by Proposition \ref{cyclic cohomology}),
Lemma
\ref{low-degree} says that we may extend the sequence of low-degree terms to
$$
\HH^1(G_i,\HH^1(\mu_r,\mathbbm k^{\times}))=\HH^1(G_i,\mathbb Z/r)\rightarrow
\HH^3(G_i,\mathbbm k^{\times})\rightarrow\HH^3(E_i,\mathbbm k^{\times})
$$
Note that we \textit{cannot} conclude that $\HH^1(G_i,\mathbb Z/r)=\Hom(G_i,\mathbb Z/r)$, since the action
may be non-trivial. But we \textit{can} conclude that it must be annihilated by both $\left|G_i\right|$ and
$r$ (Proposition \ref{group cohomology torsion}), and so since $\gcd(r,\left|G_i\right|)=1$, we see that it is trivial
and hence we have injectivity of the inflation map.

Therefore we get an exact sequence, and it splits
since both groups are annihilated by coprime integers. Finally, since the stacky structure of
$\mathcal C$ is coprime to $r$,
we see that $\mathbf R^1f_*\mathbb\mu_r=0$  by Proposition \ref{Meier mu_r}
and Proposition \ref{cyclic cohomology}, and so
we may replace $\HH^1(\mathcal C,\mathbb Z/r)$ with $\HH^1(C,\mathbb Z/r)$
(by the sequence of low-degree terms for the Leray spectral sequence for $f:\mathcal C\rightarrow C$
and $\mu_r$).

To see (5), we assume that the sequence is right-exact and that we have sections
$$
\begin{tikzcd}
\HH^2(G_i,\mathbbm k^{\times})\arrow[r, "\pi^*"] & \HH^2(E_i,\mathbbm k^{\times})\arrow[l, dashed,
bend right].
\end{tikzcd}
$$
on each fiber, and we seek to show the sequence splits. But we simply observe there is an induced
splitting from the following diagram
$$
\begin{tikzcd}
\HH^2(\mathcal C,\mathbb G_m)\arrow[r,"\cong"]\arrow[d, hook]
& \displaystyle\bigoplus_{i=1}^n\HH^2(G_i,\mathbbm k^{\times})\arrow[d]\\
\HH^2(\mathcal G,\mathbb G_m)\arrow[u, bend left, dashed, "\exists"]\arrow[r] &
\displaystyle\bigoplus_{i=1}^n\HH^2(E_i,\mathbbm k^{\times})\arrow[u, bend right, dashed].
\end{tikzcd}
$$
\end{proof}

\begin{remark}
We give here a quick description of the left map.

Let $(\alpha_i)\in\bigoplus_{i=1}^n\HH^2(G_i,\mathbbm k^{\times})$.
Since
$$
\HH^2(\mathcal C,\mathbb G_m)=\bigoplus_{i=1}^n\HH^2(G_i,\mathbbm k^{\times})
$$
we know that there is a unique $\mathbb G_m$-gerbe over $\mathcal C$, which we will
call $\mathcal H$, whose
restriction to each $BG_i$ is the $\mathbb G_m$-gerbe $\alpha_i$. Then
the image of $(\alpha_i)$ in $\HH^2(\mathcal G,\mathbb G_m)$ is just
$\pi^*([\mathcal H])$, i.e. $[\mathcal G\times_{\mathcal C}\mathcal H]$.
This is a $\mathbb G_m$-gerbe which is trivial away from the $BE_i$
and whose restriction to each $BE_i$ is given by $\inf\alpha_i\in\HH^2(E_i,\mathbbm k^{\times})$.
\end{remark}

\begin{corollary}\label{weak smooth corollary}
Suppose that $\mathcal G$, $\mathcal C$, and $C$ are as above. Then if
$\mathcal C$ is smooth:
\begin{enumerate}
\item $\mathcal G$ is automatically a root gerbe,
\item the inflation map is
automatically injective, and
\item we have
$$
\HH^2(\mathcal G,\mathbb G_m)=\HH^1(\mathcal C,\mathbb Z/r).
$$
\end{enumerate}
\end{corollary}

\begin{proof}
By Theorem \ref{stacky Tsen's} $\HH^2(\mathcal C,\mathbb G_m)=0$, which implies
that $\mathcal G$ is a root gerbe. Now on each stabilizer we have the following exact sequence
$$
0\rightarrow\mu_r\rightarrow E_i\rightarrow G_i\cong\mathbb Z/n_i\rightarrow0.
$$
Since $\mathcal G$ is a $\mu_r$-banded gerbe the extension is central. Therefore $E_i$ is abelian
-- as it's a central extension of a cyclic group by a cyclic group -- 
and thus must have a direct summand
$\mathbb Z/N_i$ which surjects onto $\mathbb Z/n_i$.
By Propositions \ref{cyclic injection} and \ref{direct sum injection} we then see that
$$
\HH^3(\mathbb Z/n_i,\mathbbm k^{\times})\hookrightarrow\HH^3(\mathbb Z/N_i,\mathbbm k^{\times})\hookrightarrow
\HH^3(E_i,\mathbbm k^{\times}),
$$
and so $\HH^3(\mathcal C,\mathbb G_m)$ injects into $\HH^3(\mathcal G,\mathbb G_m)$.
Therefore
$$
\HH^2(\mathcal G,\mathbb G_m)=\HH^1(\mathcal C,\mathbb Z/r).
$$
\end{proof}

The following is just a combination of the two previous statements, but is of enough importance
to state on its own. See \cite[Proposition 7.10]{Ach24}.
\begin{corollary}
In the case where $\mathcal C$ is smooth and $\gcd(r,\left|G_i\right|)=1$ for
$i=1,\dots,n$, we have
$$
\HH^2(\mathcal G,\mathbb G_m)=\HH^1(C,\mathbb Z/r).
$$
\end{corollary}

\newpage

\appendix 
\section{Assorted background}\label{appendix}

This paper relies very heavily on the careful application of both spectral
sequences and group cohomology. Therefore we consider it appropriate to give a brief explanation
of the properties and results invoked above.

\subsection{Group cohomology}
Most of this can be found in a standard reference, such as \cite[Chapter 17]{DF04}
or \cite[Chapter 6]{Wei94}.

For a group $G$ and an abelian group $A$ on which $G$ acts additively
(called a \textit{$G$-module}),
define the \textit{invariant subgroup}
$$
A^G=\{a\in A: ga=a,\ \forall g\in G\}.
$$
Then the cohomology of $G$ with coefficients in $A$, $\HH^n(G,A)$, is the right derived functor
$\mathbf R^n(-^G)(A)$.

\begin{proposition}\label{group cohomology torsion}
Let $G$ be a group and $A$ a $G$-module. If $G$ is $m$-torsion, then so is
$\HH^i(G,A)$ for all $i\geq 1$. If $A$ is $n$-torsion, then so is $\HH^i(G,A)$ for all
$i\geq 0$.
\end{proposition}

\begin{proposition}\label{cyclic cohomology}
 Let $\mathbbm k$ be an algebraically closed field. Let $\mu_r$ act trivially
 on $\mathbbm k^{\times}$. Then
 $$
 \HH^{2k-1}(\mu_r,\mathbbm k^{\times})=\Hom(\mu_r,\mathbbm k^{\times})=\mathbb Z/r
 $$
 and
 $$
 \HH^{2k}(\mu_r,\mathbbm k^{\times})=0
 $$
 for $k\geq 1$.
\end{proposition}

\begin{proposition}\label{cyclic injection}
Let $\mathbb Z/n$ and $\mathbb Z/mn$ act trivially on $\mathbbm k^{\times}$. Then
$$
\inf:\HH^3(\mathbb Z/n,\mathbbm k^{\times})\rightarrow
\HH^3(\mathbb Z/mn,\mathbbm k^{\times})
$$
is given by the canonical injection $1\mapsto m$.
\end{proposition}

\begin{proof}
More generally, the map
$$
\HH^k(\mathbb Z/n,\mathbbm k^{\times})\rightarrow
\HH^k(\mathbb Z/mn,\mathbbm k^{\times})
$$
is given by $1\mapsto m^{\left \lfloor{\frac k2}\right \rfloor }$ \cite[XIII Ex. 1.1]{Ser79},
which for $k=3$ is $1\mapsto m$.
\end{proof}

\begin{proposition}\label{direct sum injection}
If $\mathbb Z/a$ and $\mathbb Z/b$ both act trivially on $\mathbbm k^{\times}$, then
$$
\HH^k(\mathbb Z/a,\mathbbm k^{\times})\hookrightarrow\HH^k(\mathbb Z/a\oplus\mathbb Z/b,
\mathbbm k^{\times})
$$
for all $k\geq 1$.
\end{proposition}

\subsection{Spectral Sequences}
A standard reference for this section is \cite[Chapter 5]{Wei94}.

\begin{definition}
A \textit{spectral sequence} consists of the following data
\begin{enumerate}
\item \textit{sheets} $E_r$ for $r\geq r_0$,
\item for all integers $p,q$, objects $E_r^{p,q}$, and
\item maps $d_r:E_r^{p,q}\rightarrow E_r^{p+r,q-(r-1)}$ such that $d_r\circ d_r=0$ and
$E_{r+1}^{p,q}$ is the homology of $E_r$ at $(p,q)$.
\end{enumerate}
\end{definition}

In practice, many spectral sequences (including the ones used in this paper), begin at the $E_2$ sheet.
One then generally uses the $E_2$ sheet to calculate the subsequent sheets.

In this section we will
assume that all spectral sequences are \textit{first quadrant}, meaning the only non-zero terms have
$p,q\geq0$. Note that any first quadrant spectral sequence must eventually stabilize in the sense that there
is some $N$ such that
$E_r^{p,q}=E_N^{p,q}$ for all $r\geq N$. We call this stable value $E_{\infty}^{p,q}$.

We say that a first quadrant spectral sequence \textit{converges} to $\HH^*$ for some collection of objects
$(\HH^n)$ if for each $n$ there is a filtration
$$
0\subseteq F^n\HH^n\subseteq F^{n-1}\HH^n\subseteq\dots\subseteq F^0\HH^n=\HH^n
$$
such that
$$
E_{\infty}^{p,q}\cong F^p\HH^{p+q}/F^{p+1}\HH^{p+q}.
$$
When such a filtration exists, we denote this by
$$
E_r^{p,q}\implies\HH^{p+q}.
$$

Convergent first quadrant spectral sequences always have a \textit{sequence
of low-degree terms}.

\begin{lemma}[Sequence of low-degree terms]\label{low-degree}
If
$$
E_2^{p,q}\implies\HH^{p+q}(X)
$$
is a first quadrant spectral sequence, then there is an exact sequence
$$
0 \rightarrow E_2^{1,0}\rightarrow\HH^1(X)\rightarrow E_2^{0,1}\rightarrow
E_2^{2,0}\rightarrow\ker\left[\HH^2(X)\rightarrow E_2^{0,2}\right]\rightarrow
E_2^{1,1}\rightarrow E_2^{3,0}.
$$
This can be extended an additional term to $\HH^3(X)$ if $E_2^{0,2}=0$.
\end{lemma}

\begin{proof}
This is standard, but we will give an explanation of the extension to $\HH^3(X)$.
The filtration on $\HH^3(X)$ associated to the spectral sequence is
$$
0\subseteq F^3\HH^3\subseteq F^2\HH^3\subseteq F^1\HH^3\subseteq F^0\HH^3=\HH^3(X),
$$
where $F^i\HH^3/F^{i+1}\HH^3\cong E_{\infty}^{i,3-i}$.
Therefore there is a canonical injection $E_{\infty}^{3,0}\hookrightarrow\HH^3(X)$.
Exactness of our desired extension says that we need a morphism $E_2^{3,0}\rightarrow\HH^3(X)$
whose kernel is precisely the image of $E_2^{1,1}$.

Note that
$$
E_3^{3,0}=\frac{\ker(E_2^{3,0}\rightarrow E_2^{5,-1}=0)}{\im(E_2^{1,1}\rightarrow E_2^{3,0})}
$$
and so we have exactness of the extension as long as $E_3^{3,0}=E_{\infty}^{3,0}$. The obstruction
to this happening is whether or not $E_4^{3,0}=E_3^{3,0}$. Since
$$
E_4^{3,0}=\frac{\ker(E_3^{3,0}\rightarrow E_3^{6,-2}=0)}{\im(E_3^{0,2}\rightarrow E_3^{3,0})},
$$
we see that $E_4^{3,0}=E_3^{3,0}$ if and only if the image of
$E_3^{0,2}$ is trivial in $E_3^{3,0}$. In particular, we get equality and hence
a further extension of the sequence if $E_2^{0,2}=0$.
\end{proof}

We will make
extensive use of the Grothendieck spectral sequence, the Leray spectral sequence,
and the Lyndon-Hochschild-Serre spectral sequence.
Our statement of the Grothendieck spectral sequence is only for the context in which we will use it (pushforward
of abelian sheaves on stacks), though it applies in much larger generality.

\begin{definition}
For morphisms $f:\mathcal X\rightarrow\mathcal Y$ and $g:\mathcal Y\rightarrow\mathcal Z$
of stacks and an abelian sheaf $F$ on $\mathcal X$,
the \textit{Grothendieck spectral sequence} is
$$
E_2^{p,q}=\mathbf R^pg_*(\mathbf R^qf_*F)\implies
\mathbf R^{p+q}(g\circ f)_* F.
$$
\end{definition}

\begin{definition}
For a morphism $f:\mathcal X\rightarrow\mathcal Y$ of stacks and an abelian sheaf $F$ on $\mathcal X$, the
\textit{Leray spectral sequence} is
$$
E_2^{p,q}=\HH^p(\mathcal Y,\mathbf R^qf_*F)\implies\HH^{p+q}(\mathcal X,F).
$$
\end{definition}

\begin{definition}
For a group $G$ with normal subgroup $N$ and $G$-module $A$, the \textit{Lyndon-Hochschild-Serre}
spectral sequence is
$$
E_2^{p,q}=\HH^p(G/N,\HH^q(N,A))\implies\HH^{p+q}(G,A).
$$
\end{definition}

\newpage

\newcommand{\etalchar}[1]{$^{#1}$}

\end{document}